\documentclass[smallextended]{svjour3}       
\smartqed  
\usepackage{graphicx}
\usepackage{dsfont}
\usepackage{color,xcolor,multirow}
\newcommand{\jf}[1]{\textcolor{green}{#1}}
\newcommand{\yl}[1]{\textcolor{red}{#1}}

 \usepackage{graphicx,url}
\usepackage{amsmath,enumerate}
 \usepackage{amssymb}
 \usepackage{bm}
 \usepackage{algorithmic,algorithm}

\usepackage{url}
\usepackage{diagbox}
\newcommand{\jfa}[1]{\textcolor{magenta}{#1}}
\newcommand{\off}[1]{}

\begin{document}

\title{Minimizing Quotient Regularization Model
}


\author{Chao Wang         \and
       Jean-Francois Aujol \and
        Guy Gilboa \and 
        Yifei Lou 
}


\institute{C. Wang \at
             Department of Statistics and Data Science, Southern University of Science and Technology, Shenzhen 518005, Guangdong Province, China  \\
             National Centre for Applied Mathematics Shenzhen, Shenzhen 518055, Guangdong Province, China \\
              \email{wangc6@sustech.edu.cn}           
           \and
           J. Aujol \at
University of Bordeaux, Bordeaux INP, CNRS, IMB, UMR 5251, F-33400 Talence, France \\
\email{jean-francois.aujol@math.u-bordeaux.fr}
\and 
G. Gilboa \at 
Technion, Israel Institute of Technology, Haifa, Israel\\
\email{guy.gilboa@ee.technion.ac.il} 
\and 
Y. Lou \at
Mathematics Department, University of North Carolina at Chapel Hill, Chapel Hill, NC, 27599 USA \\
\email{yflou@unc.edu}
}

\date{Received: date / Accepted: date}

\maketitle

\begin{abstract}
 Quotient regularization models (QRMs)  are a class of  powerful regularization techniques that have gained considerable attention in recent years, due to their ability to handle complex and highly nonlinear data sets. However, the nonconvex nature of QRM poses a significant challenge in finding its optimal solution.
We are interested in scenarios where both the numerator and the denominator of QRM are absolutely one-homogeneous functions, which is widely applicable in the fields of signal processing and image processing. 
In this paper, we
utilize a gradient flow to minimize such QRM in combination with a quadratic data fidelity term.  Our scheme involves solving a convex problem iteratively. 
The convergence analysis is conducted on a modified scheme in a continuous formulation, showing the convergence to a stationary point. Numerical experiments 
demonstrate the effectiveness of the proposed algorithm in terms of accuracy, outperforming the state-of-the-art QRM solvers.
\end{abstract}

\section{Introduction}

In this paper, we consider a generalized quotient regularization model (QRM) with a least-squares data fidelity term  weighted by a positive constant $\lambda$, i.e., 
\begin{equation}
    \label{eq:quotient_model}
    \min_{u\in\Omega} \dfrac {J(u)}{H(u)} + \frac \lambda 2 \|Au-f\|_2^2,
\end{equation}
where both functionals $J(\cdot), H(\cdot)$ are proper, convex, lower semi-continuous (lsc), and absolutely one-homogeneous on a proper domain $\Omega\subset  \mathbb R^n$. An absolutely one homogeneous functional $F: u\in \Omega\rightarrow \mathbb R$ satisfies $F(\alpha u) = |\alpha| F(u), \forall \alpha \in \mathbb R, u\in \Omega.$  
This definition implies that $J(u)\geq 0, J(0)=0.$ 
 We further assume by convention $\frac{J(0)}{H(0)}:=0,$ 
 thus it is well-defined at $0.$ 
The least-squares misfit between the linear operator $A$ and the measurements $f$ is a standard data fidelity term when the noise $Au-f$ is subject to the Gaussian distribution.  For other noise types, the data fidelity term is formulated differently.
We give three specific signal and image processing examples that fit into our general model \eqref{eq:quotient_model}. 

\rm{\bf Example 1} {\it ($L_1/L_2$ sparse signal recovery).} The ratio of the $L_1$ and $L_2$ norms  was prompted as a scale-invariant surrogate to the $L_0$ norm for sparse signal recovery \cite{hoyer2004non,hurley2009comparing}. Defining $J(0)/H(0)=0$ aligns with the $L_0$ norm of the zero vector. 
Recently, a constrained minimization problem was formulated, i.e., 
\[
\min_{u\in\mathbb R^{n}} \frac{\|u\|_1}{\|u\|_2} \quad \mbox{s.t.} \quad Au=f,
\]
for the ease of analyzing the theoretical properties of the $L_1/L_2$ model \cite{rahimi2019scale,xu2021analysis} as well as deriving a numerical algorithm \cite{wang2020accelerated}. Here we adopt the unconstrained formulation \cite{tao2022minimization} that is aligned with our generalized model \eqref{eq:quotient_model}
\begin{equation}\label{eq:l1/l2}
    \min_{u\in\mathbb R^{n}} \frac{\|u\|_1}{\|u\|_2} + \frac{\lambda}{2}\|Au -f \|_2^2.
\end{equation}
A more general ratio of $L_p$ over $L_q$ (quasi-)norms for $p\in (0,2)$ and $q\geq 2$ was explored in  \cite{cherni2020spoq}.

\begin{figure}[t]
		\begin{center}
			\begin{tabular}{cc}
			   $L_1$ &   $L_0$ \\
 				\includegraphics[width=0.42\textwidth]{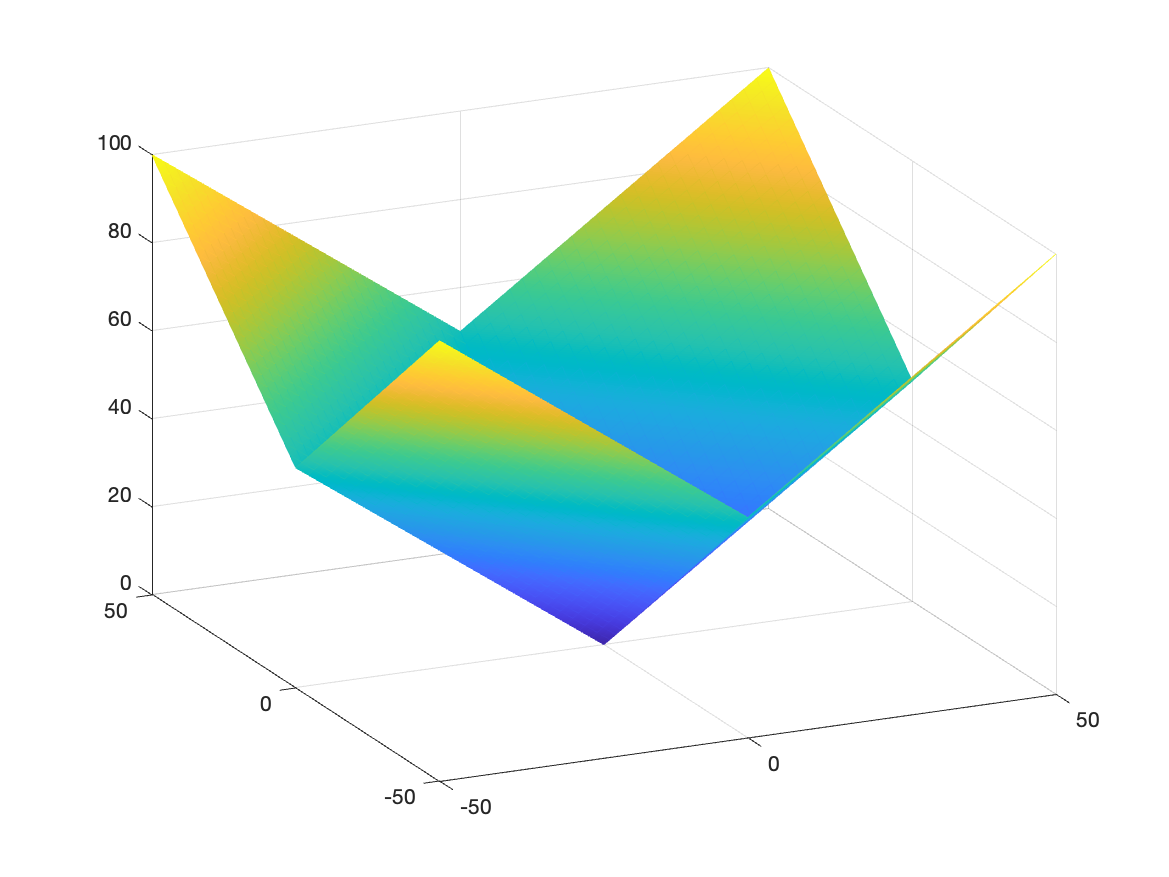} &
			    \includegraphics[width=0.42\textwidth]{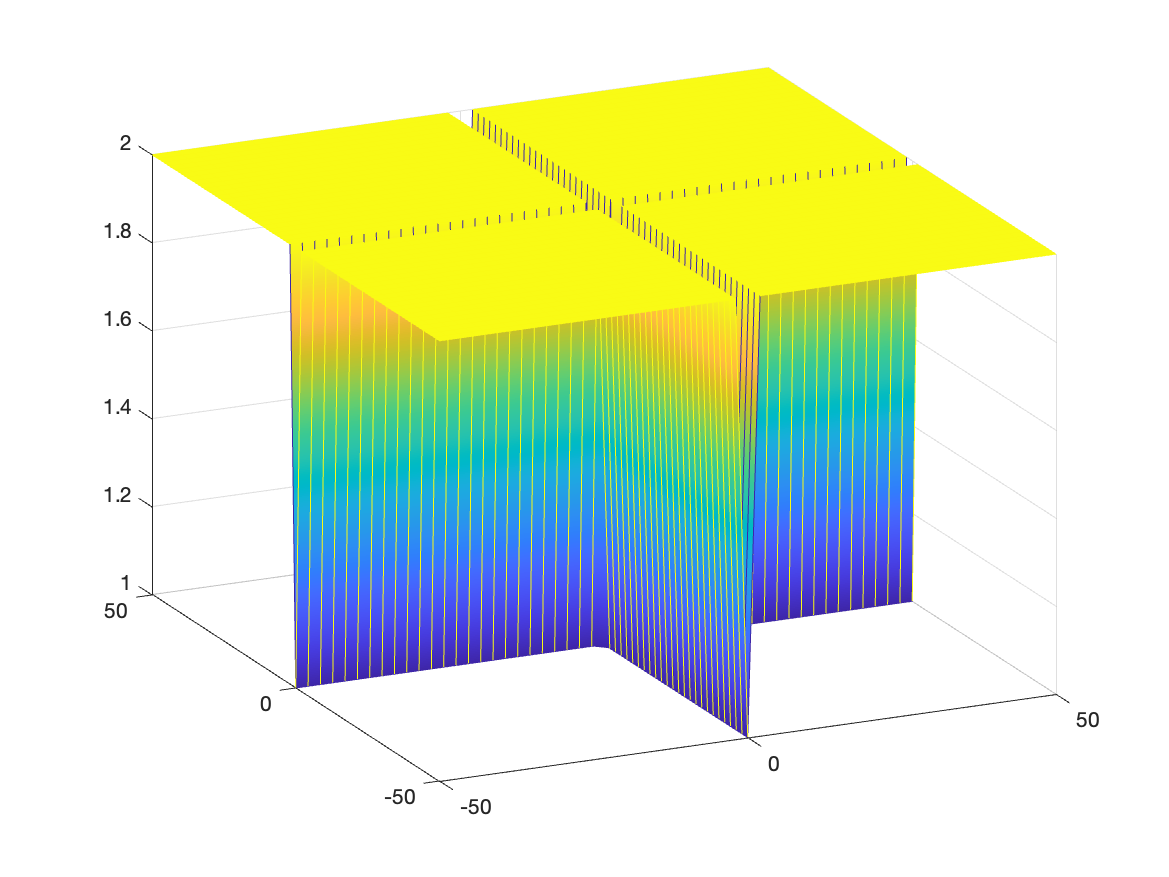}  \\
       $L_1/L_2$ &   $L_1/S_1$ \\
       \includegraphics[width=0.42\textwidth]{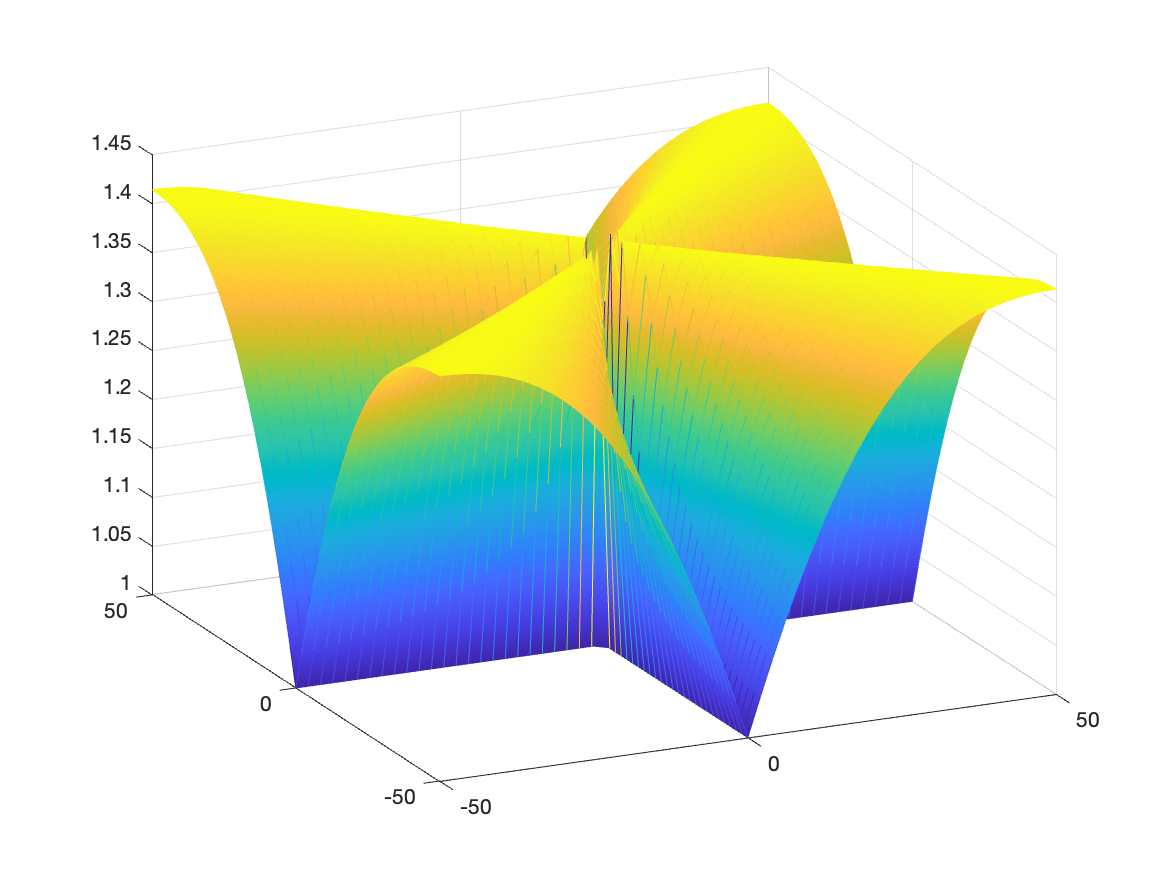} &		    \includegraphics[width=0.42\textwidth]{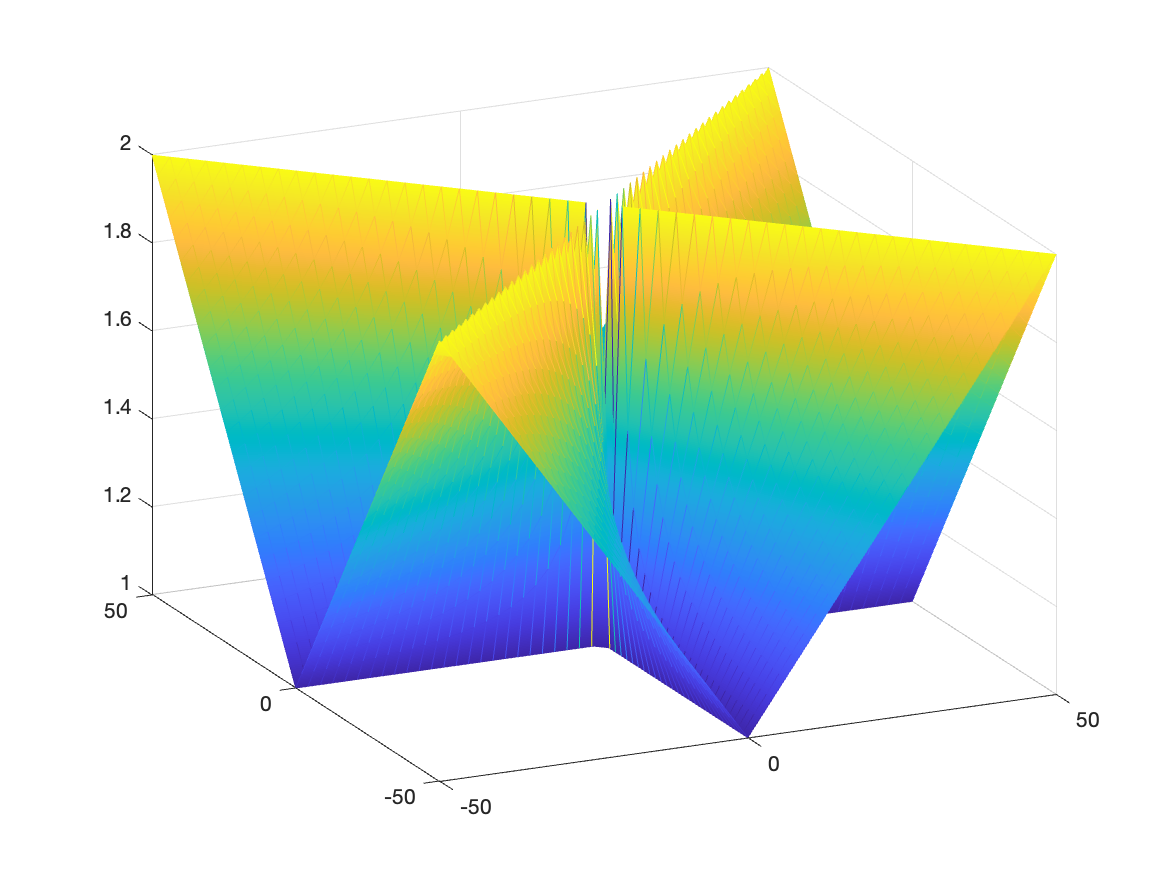}
			\end{tabular}
		\caption{A 2D illustration of $L_1/L_2$ and $L_1/S_1$ that give a better approximation to the $L_0$ norm with a comparison to the convex $L_1$ norm. 
		}\label{fig:profile}
  \end{center}
	\end{figure}

 \rm{\bf Example 2} {\it ($L_1/S_K$  sparse signal recovery).} Motivated by the truncated $L_1$ regularization (a.k.a partial sum) \cite{hu2012fast,oh2015partial} and the $L_1/L_2$ model,
 Li et al.~\cite{li2022proximal} proposed the ratio of 
the $L_1$ norm and $K$-largest sum as a sparsity-promoting regularization with a given integer $K$. When $K=1,$ it becomes the  $L_1$ norm over the infinity norm \cite{demanet2014scaling,wang2022wonderful}. For $K=n$ (the ambient dimension of $u$), $L_1/S_K$ is equivalent to $L_1/L_2$.  In Figure~\ref{fig:profile}, we use a 2D example to illustrate that
 both $L_1/L_2$ and $L_1/S_K$  can promote sparsity by approximating the $L_0$ norm.  Both ratios give a better approximation to the $L_0$ norm compared to the convex $L_1$ norm, which is largely attributed to the scale-invariant property of the $L_0$ norm and the two ratio models. 
 
Define $J(u)=\|u\|_1$  and $H(u)$ as the sum of the $K$-largest absolute values of entries, 
  denoted  as $\|u\|_{(K)}$. As both $J(\cdot)$ and $H(\cdot)$ are absolutely one-homogeneous, we consider the following problem 
\begin{equation}\label{eq:l1/Sk}
    \min_{u\in\mathbb R^{n}} \frac{\|u\|_1}{\|u\|_{(K)}} + \frac{\lambda}{2}\|Au -f \|_2^2,
\end{equation}
as a special case of \eqref{eq:quotient_model}. Note that 
the $L_1/S_K$ regularization was formulated in 
 \cite{li2022proximal} as
\begin{equation}\label{eq:l1/Sk-frac}
    \min_{u\in\mathbb R^{n}} \dfrac{\|u\|_1+\frac{\lambda}{2}\|Au -f \|_2^2}{\|u\|_{(K)}},
\end{equation}
so that a fractional programming (FP) strategy \cite{zhang2022first} can be applied. We demonstrate in our experiments that \eqref{eq:l1/Sk} outperforms \eqref{eq:l1/Sk-frac} in terms of sparse recovery.

 \rm{\bf Example 3} {\it ($L_1/L_2$ on the gradient for  image recovery).} In \cite{wang2022minimizing,wang2021limited}, the $L_1/L_2$ functional was applied to the image gradient and combined with the least-squares term,
\begin{equation}\label{eq:l12-grad}
    \min_u \frac{\|\nabla u\|_1}{\|\nabla u\|_{2}} + \frac{\lambda}{2}\|Au -f \|_2^2.
\end{equation}
Specifically, Wang et al.~\cite{wang2021limited} demonstrated that this model \eqref{eq:l12-grad} yields significant improvements in a limited-angle CT reconstruction problem. With an additional $H^1$-semi norm to \eqref{eq:l12-grad} for smoothing, a  segmentation model was proposed in \cite{wu2022efficient}. A  modification of replacing the gradient operator $\nabla$ in \eqref{eq:l12-grad} by a nonnegative diagonal matrix was explored in  \cite{lei2022physics} for electrical capacitance tomography.
 


Without the data fitting term, our model \eqref{eq:quotient_model} reduces to Rayleigh quotient problems, defined by
\begin{equation}\label{eq:rayleigh-quotient}
    \min_{u\in \Omega} R(u) :=\dfrac{J(u)}{H(u)}.
\end{equation}
The classic Rayleigh quotient problem in linear eigenvalue analysis \cite{horn2012matrix} is defined by
\begin{equation}\label{eq:eigenRayleigh}
    \min_{u\in \mathbb R^n}\dfrac {\langle u, Lu\rangle}{\|u\|_2^2},
\end{equation}
with a symmetric matrix $L\in \mathbb R^{n\times n}$. Any critical point of \eqref{eq:eigenRayleigh} is an eigenvector of the matrix $L.$ One can replace the linear mapping $Lu$ in \eqref{eq:eigenRayleigh} by a nonlinear function, thus leading to a nonlinear eigenproblem. Nossek and Gilboa \cite{nossek2018flows} proposed a continuous  flow that minimizes \eqref{eq:rayleigh-quotient} when $J(\cdot)$ is absolutely one homogeneous and $H(\cdot)$ is the square $L_2$ norm. The convergence proof was later provided in \cite{aujol2018theoretical}. Under the same setting, a nonlinear power method was proposed in \cite{bungert2021nonlinear} with connections to proximal operators and neural networks. 
 For the case when $J$ is the total variation (TV) and $H$ is the $L_1$ norm, the Rayleigh quotient 
\eqref{eq:rayleigh-quotient}  approximates the Cheeger cut problem \cite{hein2010inverse,bresson2012convergence}.  
 The quotient minimization \eqref{eq:rayleigh-quotient} also appears in learning parameterized regularizations \cite{benning2016learning} and filter functions \cite{benning2017learning}. 


In this paper, we propose a novel scheme to minimize the general model \eqref{eq:quotient_model} based on
a gradient descent flow  for the Rayleigh quotient minimization \cite{feld2019rayleigh}. We then apply the proposed algorithm to the three specific examples ($L_1/L_2,$ $L_1/S_K,$ and $L_1/L_2$ on the gradient). In each case, our algorithm requires minimizing an $L_1$-regularized subproblem, which can be solved efficiently using the alternating direction method of the multiplier (ADMM) \cite{boyd2011distributed,gabay1976dual}. 
Our analysis for the proposed algorithm is towards a slightly modified scheme. We establish a subsequential convergence of the modified scheme under the uniform boundedness of the sequence. With some additional assumptions,  the uniform bound can be proven using a continuous flow formulation. In experiments, we demonstrate the efficiency of the proposed algorithm over the relevant methods in the literature. In summary, the novelties of this paper are threefold:
 \begin{enumerate}
     \item We consider a general model \eqref{eq:quotient_model} that combines the Rayleigh quotient as a regularization with a data fidelity term. Our model has a variety of applications, especially in signal and image reconstruction. 
     \item We propose a unified algorithm with numerical insights on convergence and the solution's boundedness. 
\item Our approach can be adapted to three case studies: \eqref{eq:l1/l2}, \eqref{eq:l1/Sk}, and \eqref{eq:l12-grad}. In each case, the proposed scheme outperforms the relevant algorithms in the literature in terms of accuracy.
 \end{enumerate}

 The rest of the paper is organized as follows. Section~\ref{sect:algorithm} describes the proposed algorithms in detail, including numerical formulation and specific closed-form solutions for the three case studies. We provide mathematical analysis on the numerical scheme in Section~\ref{sect:analysis}. Extensive experiments are conducted in Section~\ref{sect:experiments} for applications  in signal and image recovery. Finally, conclusions and future works are given in Section~\ref{sect:conclusion}.

\section{Proposed algorithms}\label{sect:algorithm}
Recall that we aim at the minimization problem 
\begin{equation}\label{eq:general-obj}
    \min_u G(u) := R(u) + \frac{\lambda}{2}\|Au -f \|_2^2,
\end{equation}
with $R(u) = J(u)/H(u).$

\begin{theorem} \label{thm:u-neq0}
    Suppose $A$ is an under-determined matrix, $f \in \mathrm{Im}(A)$, and $R(\cdot)$ has an upper bound, i.e., $R(u)\leq M$. For a sufficiently large parameter $\lambda$,  the optimal solution of (\ref{eq:general-obj}) can not be ${0}.$
\end{theorem}
\begin{proof}
As $A$ is an under-determined matrix and $f \in \mathrm{Im}(A)$, there exist infinitely many solutions satisfying $Au=f,$ among which we 
    denote $\hat u$ to be the least norm solution, that is,
    \begin{eqnarray*}
        \hat u = \arg\min_u \|u\|_2 \quad \mbox{such that} \quad Au=f.
    \end{eqnarray*}
    It is straightforward that $ G(\hat u) = R(\hat u)\leq M$ and $G(0) = \frac{J(0)}{H(0)}+\frac{\lambda} 2\|f\|_2^2$. If $\lambda > \frac{2M}{\|f\|_2^2},$ then
   we have $G(\hat u)< G(0),$ which implies that 0 cannot be the global solution to \eqref{eq:general-obj}. 
\end{proof}
{\bf Remark: } Note that all the examples listed in the introduction section satisfy the boundedness assumption of $R(\cdot)$. Taking $L_1/L_2$ for an example, one has $\frac{\|u\|_1}{\|u\|_2}\leq \sqrt n$ for $u\in \mathbb R^n.$

\medskip

One  classic method to minimize $G(u)$ is by using a gradient descent flow, i.e., 
\begin{equation}
    \label{flow}
    u_t = -\nabla G(u).
\end{equation}
The derivative of $G$ can be expressed as
\begin{equation} \label{eq_nablaG}
\begin{split}
        \nabla G(u) &= \frac{H(u)p-J(u)q}{H^2(u)} + \lambda A^T(Ax-f)\\
        & =  \frac{p-R(u)q}{H(u)} + \lambda A^T(Au-f), 
\end{split}
\end{equation}
where $q\in \partial H(u), p \in \partial J(u).$ We consider the subgradient $\partial$ here as $J(\cdot), H(\cdot)$ are not necessarily differentiable. 
Plugging the gradient expression \eqref{eq_nablaG} into the flow \eqref{flow} yields 
\[
u_t = \frac{R(u)}{H(u)}q- \frac{p}{H(u)} - \lambda A^T(Au-f),
\]
which can be discretized by the iteration count $k$,
\begin{equation}\label{eq:discrete-scheme}
     \frac{u^{k+1}-u^{k}}{dt} = \frac{R(u^k)}{H(u^k)}q^k-\frac{p^{k+1}}{H(u^k)} -\lambda A^T(Au^{k+1}-f). 
\end{equation}
Note that we consider a semi-implicit  scheme in \eqref{eq:discrete-scheme} such that the update of $u^{k+1}$ is obtained by the following optimization problem,
\begin{equation}
    \label{eq:minimization_flow}
    u^{k+1} = \arg\min_u \left\{ \frac{\beta}{2}\|u- u^k\|_2^2-\frac{R(u^k)}{H(u^k)}\langle q^k, u\rangle+\frac{J(u)}{H(u^k)}+\frac{\lambda}{2}\|Au-f\|_2^2\right\},
\end{equation}
where $\beta = \frac{1}{dt}.$ 
In what follows, we describe the detailed algorithms for $L_1/L_2$ and $L_1/S_K$ in Section~\ref{sect:signal} as well as the gradient model \eqref{eq:l12-grad} in Section~\ref{sect:image}, all based on the general scheme \eqref{eq:minimization_flow}.

\subsection{Quotient regularization  for sparse signal recovery} \label{sect:signal}
For $H(u) = \|u\|_2$ and $q\in\partial H(u)$, we get $q = \frac{u}{\|u\|_2}$ if $u\neq 0;$ otherwise $q$ is a vector with each element bounded by $[-1,1].$ 
As $J(u) = \|u\|_1,$
the minimization problem \eqref{eq:minimization_flow} at the $k$th iteration becomes 
\begin{equation}
\label{eq:signal_quotient}
    u^{k+1} = \arg\min_u \left\{ \frac{\beta}{2}\|u- u^k\|_2^2-\langle h^k, u\rangle+\frac{\|u\|_1}{\|u^k\|_2}+\frac{\lambda}{2}\|Au-f\|_2^2\right\},
\end{equation}
where $h^k=\frac{R(u^k)}{H(u^k)} q^k=\frac{\|u^k\|_1}{\|u^k\|_2^3 }u^k$.  Note that the scheme \eqref{eq:discrete-scheme} becomes degenerate if $u^k=0$,  while this turns out not to be restrictive in, as $u^k=0$ never occurs in 
our experiments. On the theoretical side,  we know from Theorem \ref{thm:u-neq0} that $0$ cannot be the minimizer of the objective function in the minimization problem \eqref{eq:minimization_flow}.

To solve for the $L_1$-regularized minimization  \eqref{eq:signal_quotient}, we introduce an auxiliary variable $y$ and consider an equivalent problem 
\begin{equation}
   \min_{u,y} \frac{\beta}{2}\|y- u^k\|_2^2-\langle h^k, y\rangle+\frac{\|u\|_1}{\|u^k\|_2}+\frac{\lambda}{2}\|Ay-f\|_2^2  \quad \mbox{s.t.} \quad u=y.
\end{equation}
The corresponding augmented Lagrangian function is expressed as,
	\begin{equation}\label{eq:augL4con}
	\mathcal{L}_k(u,y;\eta) =\frac{\beta}{2}\|y- u^k\|_2^2-\langle h^k, y\rangle+\frac{\|u\|_1}{\|u^k\|_2}+\frac{\lambda}{2}\|Ay-f\|_2^2+\frac{\rho}{2}\|u-y+\eta\|_2^2,
	\end{equation}
	where  $\eta$ is a dual variable and $\rho$ is a positive parameter. 
	Then  ADMM iterates as follows
	\begin{equation} \label{ADMM_signal}
	\left\{\begin{array}{l}
 u_{j+1}=\arg\min_{u} \mathcal{L}_k (u,y_j;\eta_j)\\
	y_{j+1}=\arg\min_{y} \mathcal{L}_k (u_{j+1},y;\eta_j)\\
	\eta_{j+1} =  \eta_j + u_{j+1}-y_{j+1},
	\end{array}\right.
	\end{equation}
 where the subscript $j$ represents the inner loop index, as opposed to
	the superscript $k$ for outer iterations \eqref{eq:minimization_flow}.
	The $u$-subproblem has a closed-form solution:
 \begin{equation*}
     u_{j+1}=\mathrm{shrink}\left(y_j-\eta_j,\ \frac{1}{\rho \|u^k\|_2}\right).
 \end{equation*}
The update of $y$ follows the computation of gradient of $\mathcal{L}_k$ with respect to $y$: 
\begin{equation}\label{eq:y-update}
    y_{j+1} = (\lambda A^T A +(\beta+\rho) I )^{-1} (\beta u^k+h^k+\lambda A^Tf+\rho(u_{j+1}+\eta_j)),
\end{equation}
which involves solving a large linear system. 
In the case of sparse signal recovery when the system matrix $A\in\mathbb{R}^{m\times n}$ is under-determined, i.e., $m\ll n$, the closed-form solution of $y$ can be written in an efficient way by the Sherman–Morrison–Woodbury formula:
\begin{equation*}
    y_{j+1} = \left[\kappa I -  \lambda\kappa^2A^T \left(I+ \lambda\kappa AA^T\right)^{-1}A \right]  \left[\beta u^k+h^k+ \lambda A^Tf+\rho(u_{j+1}+\eta_j)\right]. 
\end{equation*}
where $\kappa = 1/(\beta+\rho)$ and  the matrix $I+ \lambda \kappa AA^T$ is in $m$-by-$m$ size, which is much smaller than inverting an $n\times n$ matrix in \eqref{eq:y-update}. Using the Choleskey decomposition for $I+ \lambda \kappa AA^T$ can further accelerate the computation.

For the $L_1/S_K$ model \eqref{eq:l1/Sk}, $H(u) = \|u\|_{(K)}$ and its subgradient is a random vector bounded by $[-1,1]$ if $u=0.$ In addition, when $u\neq 0,$ one has 
\[
q_i = 
    \begin{cases}
       \frac {u_i}{\|u\|_{(K)}} &  i \in \Omega_K(u) \\
        0 & \text{Otherwise},
    \end{cases}
\]
where $q\in \partial H(u)$ and $\Omega_K(u)$ is the index set of the $K$-largest absolute values of $u$.
As a result, the algorithm for the $L_1/S_K$ model \eqref{eq:l1/Sk}
is the same as \eqref{eq:signal_quotient} except that $h^k = \frac{\|u^k\|_1}{\|u^k\|^3_{(k)}}v^k$ with
\begin{equation}
v^k_i = 
    \begin{cases}
       u^k_i &  i \in \Omega_K(u^k) \\
        0 & \text{Otherwise}.
    \end{cases}
\end{equation}

Algorithm~\ref{alg:unified_signal} presents a unified scheme that minimizes the $L_1/L_2$ and $L_1/S_K$ models with the least-squares fit.

\begin{algorithm}[t]
		\caption{Proposed algorithm for the models of $L_1/L_2$  and $L_1/S_K$. }
		\label{alg:unified_signal}
		\begin{algorithmic}[1]
			\STATE{Input: a linear operator $A$, observed data $f$}
			\STATE{Parameters: $\rho, \lambda, \beta$, $\kappa = 1/(\beta +\rho),$ kMax, jMax,   $\epsilon \in \mathds{R}$, and $K$ for the $L_1/S_K$ model}
			\STATE{Initialize: $ \eta =  0, k,j = 0$ and $u^0$}
			
			\WHILE{$k < $ kMax or $\| u^{k}- u^{k-1}\|_2/\| u^{k}\|_2 > \epsilon$}
			
			\WHILE{$j < $ jMax or $\| u_{j}-u_{j-1}\|_2/\| u_{j}\|_2 > \epsilon$}
			\STATE{$   u_{j+1}=\mathrm{shrink}\left(y_j-\eta_j,\ \frac{1}{\rho \|u^k\|_2}\right)$}
			\STATE{
   $
			y_{j+1} = \left[\kappa I -  \lambda\kappa^2A^T \left(I+ \lambda\kappa AA^T\right)^{-1}A \right]  \left[\beta u^k+h^k+ \lambda A^Tf+\rho(u_{j+1}+\eta_j)\right]$
   }
			\STATE{$\eta_{j+1} =  \eta_j + u_{j+1}-y_{j+1} $}
			\STATE{Assign $j$ by $j + 1$}
			\ENDWHILE
			\STATE{ Set $u^{k+1}$ as $ u_j$}
\STATE{Update $ h^{k+1}$ by 
$
    h^{k+1}=\begin{cases}
        \frac{\|u^k\|_1}{\|u^k\|_2^3 }u^k & \text{ for } L_1/L_2 \\
        \frac{\|u^k\|_1}{\|u^k\|^3_{(k)}}v^k & \text{ for } L_1/S_K
    \end{cases}
$
}
			\STATE{Assign $k$ and $j$ by $k+1$ and 0, respectively}
			\ENDWHILE
			\RETURN $ u^\ast =  u^{k}$ \end{algorithmic}
	\end{algorithm}

\subsection{Quotient regularization  for image recovery}\label{sect:image}
When $J(u) = \|\nabla u\|_1$ and $H(u) = \|\nabla u\|_2$, we get $q = \frac{-\Delta u}{\|\nabla u\|_2}$ if $\nabla u \neq 0;$ otherwise $q$ is a vector with each element bounded by $[-1,1].$ Hence the minimization problem \eqref{eq:minimization_flow} in the $k$-iteration becomes 
\begin{equation}
\label{eq:image_quotient} 
    u^{k+1} = \arg\min_u \left\{ \frac{\beta}{2}\|u- u^k\|_2^2-\langle h^k, u\rangle+\frac{\|\nabla u\|_1}{\|\nabla u^k\|_2}+\frac{\lambda}{2}\|Au-f\|_2^2\right\},
\end{equation}
where $h^k = \frac{\|\nabla u^k\|_1}{\|\nabla u^k\|_2^3}\Delta u^k$. 
The subproblem \eqref{eq:image_quotient} is a TV regularization with additional linear and least-squares terms, which can be solved by  ADMM. In particular, we introduce one auxiliary variable $y=\nabla u$ upon convergence, and formulate the  augmented Lagrangian function corresponding to \eqref{eq:signal_quotient} as,
	\begin{equation}\label{eq:aug}
	\mathcal{L}_k(u, y;\eta) =\frac{\beta}{2}\|u- u^k\|_2^2-\langle h^k, u\rangle+\frac{\|y\|_1}{\|\nabla u^k\|_2}+\frac{\lambda}{2}\|Au-f\|_2^2+\frac{\rho}{2}\|\nabla u-y+\eta\|_2^2,
	\end{equation}
	where $\eta$ is a dual variable and $\rho$ is a positive parameter. 
Then  ADMM iterates as follows
	\begin{equation} \label{ADMM_image}
	\left\{\begin{array}{l}
 u_{j+1}=\arg\min_{u} \mathcal{L}_k (u,y_j;\eta_j)\\
	y_{j+1}=\arg\min_{y} \mathcal{L}_k (u_{j+1},y;\eta_j)\\
	\eta_{j+1} =  \eta_j + \nabla u_{j+1}-y_{j+1}.
	\end{array}\right.
	\end{equation}
Taking the derivative of \eqref{ADMM_image} with respect to $u$, we get 
\begin{equation}\label{eq:update_u_mri}
    u_{j+1} = (\lambda A^T A-\rho \Delta + \beta I)^{-1} (\lambda A^Tf +\beta u^k+\rho(y-\eta_j)+h^k). 
\end{equation}
For image deblurring or the MRI reconstruction,  the inverse in the $u$-update \eqref{eq:update_u_mri}  can be computed efficiently via the fast Fourier transform.

The update for the variable $y$ is given by 
\begin{equation*}
    y_{j+1} = \operatorname{shrink}\left(\nabla u_{j+1} +\eta_j, \frac{1}{\rho \|\nabla u\|_2}\right). 
\end{equation*}
We summarize the proposed algorithm for minimizing the $L_1/L_2$ on the gradient in Algorithm~\ref{alg:unified_image}.

\begin{algorithm}[t]
		\caption{Proposed algorithm for the $L_1/L_2$ model on the gradient.}
		\label{alg:unified_image}
		\begin{algorithmic}[1]
			\STATE{Input: a linear operator $A$, observed data $f$,  }
			\STATE{Parameters: $\rho, \lambda, \beta$, kMax, jMax,  and  $\epsilon \in \mathds{R}$}
			\STATE{Initialize: $ \eta =  0, k,j = 0$ and $u^0$}
			
			\WHILE{$k < $ kMax or $\| u^{k}- u^{k-1}\|_2/\| u^{k}\|_2 > \epsilon$}
			
			\WHILE{$j < $ jMax or $\| u_{j}-u_{j-1}\|_2/ \| u_{j}\|_2 > \epsilon$}
			\STATE{$     u_{j+1} = (\lambda A^T A-\rho \Delta + \beta I)^{-1} (\lambda A^Tf +\beta u^k+\rho(y-\eta_j)+h^k)$}
			\STATE{$
			    y_{j+1} = \operatorname{shrink}\left(\nabla u_{j+1} +\eta_j, \frac{1}{\rho \|\nabla u\|_2}\right) 
			$ }
			\STATE{$\eta_{j+1} =  \eta_j + u_{j+1}-y_{j+1} $}
			\STATE{Assign $j$ by $j + 1$}
			\ENDWHILE
			\STATE{ Set $u^{k+1}$ as $ u_j$}
\STATE{Update $ h^{k+1}$ by 
$
    h^k = \frac{\|\nabla u^k\|_1}{\|\nabla u^k\|_2^3}\Delta u^k
$
}

			\STATE{Assign $k$ and $j$ by $k+1$ and 0, respectively}
			\ENDWHILE
			\RETURN $ u^\ast =  u^{k}$ \end{algorithmic}
	\end{algorithm}

\section{Mathematical analysis}\label{sect:analysis}
This section is split into two parts.
In Section~\ref{sect: convergence}, we prove the convergence of a modified scheme to the solution of the quotient model \eqref{eq:quotient_model}.
To do so, we need a technical uniform bound assumption, which is analyzed in Section~\ref{sect:bounded} based on a continuous formulation of the scheme.

\subsection{Convergence of the scheme}\label{sect: convergence}

We first show that a fully implicit version of the numerical scheme \eqref{eq:minimization_flow} converges (up to a subsequence) to a solution of our original problem \eqref{eq:quotient_model} under a reasonable uniform bound assumption.
In our analysis, we make use of Lemma~\ref{lemma_diff} that is related to the subdifferential of one homogeneous convex function (see for instance \cite{bungert2021nonlinear,feld2019rayleigh}):

\begin{lemma} \label{lemma_diff}
If $J$ is a convex one homogeneous function, then the following hold:
\begin{itemize}
    \item[(i)]
If $p \in \partial J(u)$, then $J(u)=\langle p,u \rangle$.
    \item[(ii)]
    If $p \in \partial J(u)$, then $J(v) \geq \langle p,v \rangle, \ \forall v$. 
\end{itemize}
\end{lemma}

\paragraph{Fully implicit scheme:}
We recall that the sequence $\{u^k\}$ is defined by Equation \eqref{eq:minimization_flow}.
In fact, we are going to analyze a slightly different scheme, which is referred to as a  fully implicit scheme,
\begin{equation}
    \label{eq:minimization_flow_new}
    u^{k+1} = \arg\min_u \left\{ \frac{\beta}{2}\|u- u^k\|_2^2-\frac{R(u^k)}{H(u)}\langle q^k, u\rangle + R(u)+\frac{\lambda}{2}\|Au-f\|_2^2\right\},
\end{equation}
where the term $\frac{1}{H(u^k)}$ in \eqref{eq:minimization_flow} has been replaced by $\frac{1}{H(u)}$.
We remark that the numerical scheme \eqref{eq:minimization_flow} is much easier to handle with the term $\frac{1}{H(u^k)}$, but the mathematical analysis of \eqref{eq:minimization_flow_new} happens to be much easier with $\frac{1}{H(u)}$. We establish in Theorem~\ref{th_tech} that $\|u^{k+1} -u^k\|_2 \to 0$ when $k\to +\infty$.


\begin{theorem} \label{th_tech}
For absolutely one-homogeneous functionals $J(\cdot), H(\cdot)$, 

then $\sum \|u^{k+1} -u^k\|_2^2 $ converges, and thus $\|u^{k+1} -u^k\|_2 \to 0$ when $k\to +\infty$.
\end{theorem}

\begin{proof}
Define the objective function in \eqref{eq:minimization_flow_new} by
\begin{equation}
    F(u):=\frac{\beta}{2}\|u- u^k\|_2^2-\frac{R(u^k)}{H(u)} \langle q^k, u\rangle+R(u)+\frac{\lambda}{2}\|Au-f\|_2^2.
\end{equation}
It is straightforward that
\begin{equation}
    F(u^k)=-\frac{R(u^k)}{H(u^k)} \langle q^k, u^k \rangle+R(u^k)+\frac{\lambda}{2}\|Au^k-f\|_2^2.
\end{equation}
Since $H$ is absolutely one-homogeneous, we use Lemma~\ref{lemma_diff} (i) to obtain $\langle q^k , u^k\rangle = H(u^k),$ thus leading to
    \begin{equation}\label{eq:usingLemma1}
     F(u^k)=
 - R(u^k) +R(u^k)+\frac{\lambda}{2}\|Au^k-f\|_2^2    =\frac{\lambda}{2}\|Au^k-f\|_2^2.
\end{equation}
It follows from Lemma~\ref{lemma_diff} (ii) that $H(u^{k+1})\geq \langle q^k, u^{k+1}\rangle,$ which implies that
\begin{equation*}
\begin{split}\label{eq:usingLemma2}
    F(u^{k+1}) & =  \frac{\beta}{2}\|u^{k+1}- u^k\|_2^2-\frac{R(u^k)}{H(u^{k+1})} \langle q^k, u^{k+1} \rangle+R(u^{k+1})+\frac{\lambda}{2}\|Au^{k+1}-f\|_2^2
    \\
    & \geq  \frac{\beta}{2}\|u^{k+1}- u^k\|_2^2
    -\frac{R(u^k)}{H(u^{k+1})} H(u^{k+1}) +R(u^{k+1})+\frac{\lambda}{2}\|Au^{k+1}-f\|_2^2\\
    &=\frac{\beta}{2}\|u^{k+1}- u^k\|_2^2
    -R(u^k)  +R(u^{k+1})+\frac{\lambda}{2}\|Au^{k+1}-f\|_2^2.
\end{split}  
\end{equation*}
We use the fact that $F(u^{k+1}) \leq F(u^k)$ to deduce:
\begin{equation}
    \frac{\beta}{2}\|u^{k+1}- u^k\|_2^2
    -R(u^k)  +R(u^{k+1})+\frac{\lambda}{2}\|Au^{k+1}-f\|_2^2
    \leq
    \frac{\lambda}{2}\|Au^k-f\|_2^2.
\end{equation}
Summing from $1$ to $N$, 
we get:
\begin{eqnarray*}
    \frac{\beta}{2} \sum_{k=1}^N \|u^{k+1}- u^k\|_2^2
    &\leq&
    R(u^1)-R(u^{N+1})
    +\frac{\lambda}{2}\Big(\|Au^1-f\|_2^2
    -\|Au^{N+1}-f\|_2^2 \Big)\\
    &\leq &
    R(u^1)
    +\frac{\lambda}{2}\|Au^1-f\|_2^2,
\end{eqnarray*}
due to $R(u)\geq 0$ and $\|Au-f\|_2^2\geq 0$ for any $u.$
Let $N\rightarrow \infty,$ we obtain that $\sum_{k=1}^{\infty} \|u^{k+1}- u^k\|_2^2$ is bounded, which implies that $\|u^{k+1}- u^k\|_2^2\rightarrow 0$.
\end{proof}

Now that we have proven that 
$\|u^{k+1} -u^k\|_2 \to 0$ when $k\to +\infty$, we are going to be able to pass   the limit up to a subsequence in the optimality condition of Problem~\eqref{eq:quotient_model}.

\begin{theorem} \label{th_conv}
    For absolutely one-homogeneous functionals $J(\cdot), H(\cdot)$, if the sequence $\{u_k\}$ defined by \eqref{eq:minimization_flow_new} is uniformly bounded,
then there exists a subsequence of $(u^k,p^k,q^k)$ that converges to $ (u^*,p^*,q^*)$.
    Moreover, we have
    \begin{equation}
        p^* \in \partial J(u^*), \ q^* \in \partial H(u^*), 
   \quad 
    \mbox{and}
    \quad
        0=\lambda \, A^T(Au^*-f) 
        +  \frac{p^*-R(u^*)q^*}{H(u^*)}.
    \end{equation}
\end{theorem}

\begin{proof}
Since $u^k$ is uniformly bounded, it is also the case for the subgradients $p^k$ and $q^k$.
Thus, there exists $(u^*,p^*,q^*)$ such that up to a subsequence, $(u^k,p^k,q^k) \to (u^*,p^*,q^*)$.
The optimality condition for \eqref{eq:minimization_flow_new} can be written as:
\begin{eqnarray*}
    0 & = & \beta(u^{k+1}-u^k) + \lambda \, A^T(Au^{k+1}-f) 
    \\
    & &
    +\frac{R(u^k)q^k}{H(u^{k+1})}- \frac{R(u^k)\langle q^k,u^{k+1}\rangle q^{k+1}}{(H(u^{k+1}))^2}+ \frac{p^{k+1}-R(u^{k+1})q^{k+1}}{H(u^{k+1})}.
\end{eqnarray*}
    Thanks to Theorem~\ref{th_tech}, we can pass to the limit in this last equation   to get:
    \begin{equation}\label{eq:opt-cond}
        0=\lambda \, A^T(Au^*-f) 
        +  \frac{p^*-R(u^*)q^*}{H(u^*)},
    \end{equation}
    where we use Lemma~\ref{lemma_diff} for $\langle q^*,u^* \rangle =H(u^*).$ Note that  \eqref{eq:opt-cond} is the original optimality condition $\nabla G(u^*)=0$, i.e.,  Equation~\eqref{eq_nablaG} for the optimization problem~\eqref{eq:quotient_model}.
\end{proof}

\subsection{Uniform boundedness of the sequence $\{u^k\}$}\label{sect:bounded}


The goal of this subsection is to explain why the technical assumption on the uniform boundedness of the sequence $\{u^k\}$ is reasonable for 
Theorem~\ref{th_conv}.
Instead of dealing with the discrete sequence $\{u^k\}$, we conduct our analysis in a continuous setting, which enables us to have tractable computations. In particular, we consider a differentiable function $u$ of the continuous flow, that is,     $u_t = -\nabla G(u)$ in \eqref{flow}. 
    Notice that $u^k$ defined by \eqref{eq:minimization_flow_new} can be seen as a discretized version of $u$.


We show in Theorem~\ref{thm:flow29} that a mapping of $t \mapsto \|u\|_2^2$
is a non-increasing function as long as $\|Au\|\geq \|f\|$.

\begin{theorem}
\label{thm:flow29}
Suppose $u(t)$ is a differentiable function with respect to the time $t$ that satisfies the flow \eqref{flow}, i.e., 
\begin{equation} \label{flow29}
        u_t=- \nabla G(u)= -\lambda \, A^T(Au-f) 
        -  \frac{p-R(u)q}{H(u)}.
    \end{equation}
    If $\|Au\|_2\geq \|f\|_2$,
    then
    \begin{equation}
        \frac{d}{dt}\left(
        \|u\|_2^2
        \right)\leq 0.
    \end{equation}
\end{theorem}

\begin{proof}
Simple calculations lead to
    \begin{eqnarray*}
        \frac{d}{dt}\left(
        \|u\|_2^2
        \right) & = &
        \langle u, u_t\rangle \\
        & = &
        -\lambda \langle Au-f, Au \rangle -  \frac{\langle p, u \rangle -R(u)\langle q,u\rangle }{H(u)}
        \\
        & = &
        -\lambda \langle Au-f, Au \rangle -  \frac{J(u) -R(u)  H(u) }{H(u)} 
        \\
        & = &
        -\lambda \left( \|Au\|_2^2 -\langle f, Au \rangle \right),
    \end{eqnarray*}
where we use Lemma~\ref{lemma_diff} with $p \in \partial J(u)$ and $q \in \partial H(u)$.
It further follows from the Cauchy-Schwartz inequality that
\begin{equation}
     \frac{d}{dt}\left(
        \|u\|_2^2
        \right) \leq  \lambda \|Au\| \left( \| f\| - \|Au\|\right).
\end{equation}
Consequently, if   $\|Au\|\geq \|f\|$, then $\|u\|_2^2$ is a non-increasing function. 
\end{proof}

We give a numerical verification of Theorem~\ref{thm:flow29} in Figure~\ref{fig:unorm}. 
A direct consequence of Theorem~\ref{thm:flow29} leads to the  following two corollaries. 

\begin{corollary}
If $A$ is coercive (i.e. there exists $c>0$ such that $\|Au\| \geq c \|u\|$), then any function $u$ satisfying the flow \eqref{flow29} is uniformly bounded.
\end{corollary}

The coercivity assumption on $A$ can further be weakened.
For instance, if we write $u=v+w$ with $v \in \text{Ker}(A)$ and 
$w \in (\text{Ker}(A))^\perp$ (notice that this decomposition exists and is unique), then we only need a uniform boundedness assumption on $v$.

\begin{corollary}
Suppose $u(t)$ satisfies the flow \eqref{flow29}. 
We can uniquely express $u=v+w$ with $v \in \text{Ker}(A)$ and 
$w \in (\text{Ker}(A))^\perp$.
If $v$ is uniformly bounded, then  $u$ is uniformly bounded.
    \end{corollary}

\section{Numerical Results}\label{sect:experiments}

In this section, we showcase the effectiveness of the proposed algorithms through a set of numerical experiments. All of these experiments were carried out on a typical laptop featuring a CPU (AMD Ryzen 5 4600U at 2.10GHz) and MATLAB (R2021b). 

We start with some numerical insights of the proposed scheme in Section~\ref{sect:exp-alg}, followed by case studies of sparse signal recovery  in Section~\ref{sect:exp-signal} and MRI reconstruction in Section~\ref{sect:exp-image}.   Specifically for signal recovery, 
we conduct experiments in a noisy setting, aiming to recover an underlying sparse vector $u\in\mathbb R^n$ with $s$ non-zero elements from a set of noisy measurements,  $f=Au+\nu$,
where $A\in\mathbb R^{m\times n}$ 
is a Gaussian random matrix with each column normalized by zero mean and unit Euclidean norm, and $\nu$ is  Gaussian noise  with  zero  mean and standard deviation $\sigma$. We fix the ambient dimension $n=512$, sparsity  $s=130,$ and noise level $\sigma = 0.1$, while varying 
the number of measurements $m$ to examine the performance of sparse signal recovery. Notice that fewer measurements 
result in a more challenging  recovery process. We use the mean-square error (MSE) metric to evaluate the recovery performance.  we can obtain the ordinary least square (OLS) solution if we know the ground truth of the support set of $\Lambda = \operatorname{supp}(u)$, which refers to the index set of nonzero entries in $u$. In this case, we can consider the mean squared error (MSE) of OLS as the benchmark for the oracle performance, using $\sigma \mathrm{tr}(A_{\Lambda}^{\top}A_{\Lambda})^{-1}$, where $A_{\Lambda}$ refers to a submatrix of $A$ by taking the columns corresponding to the index set $\Lambda.$

\subsection{Algorithm behaviors}\label{sect:exp-alg}

\begin{figure}
		\begin{center}
			\begin{tabular}{cc}
			   signal recovery &  image recovery \\
				\includegraphics[width=0.42\textwidth]{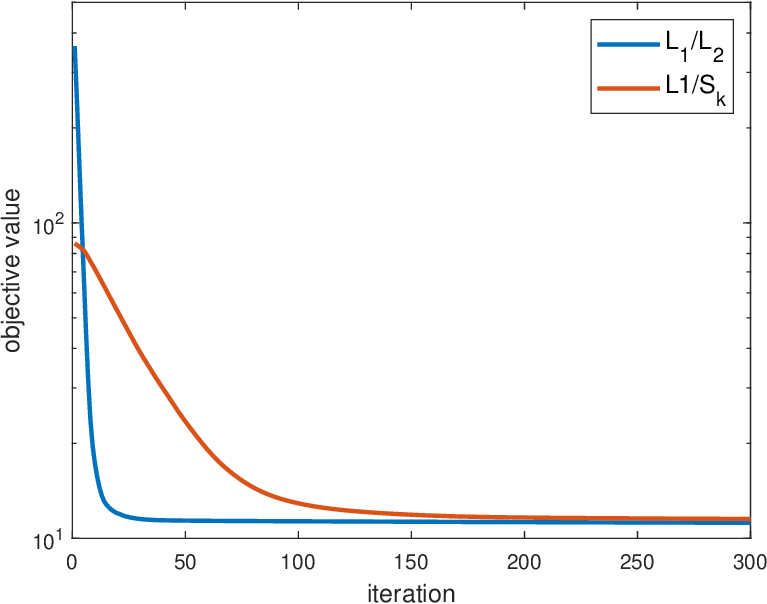} &
			    \includegraphics[width=0.42\textwidth]{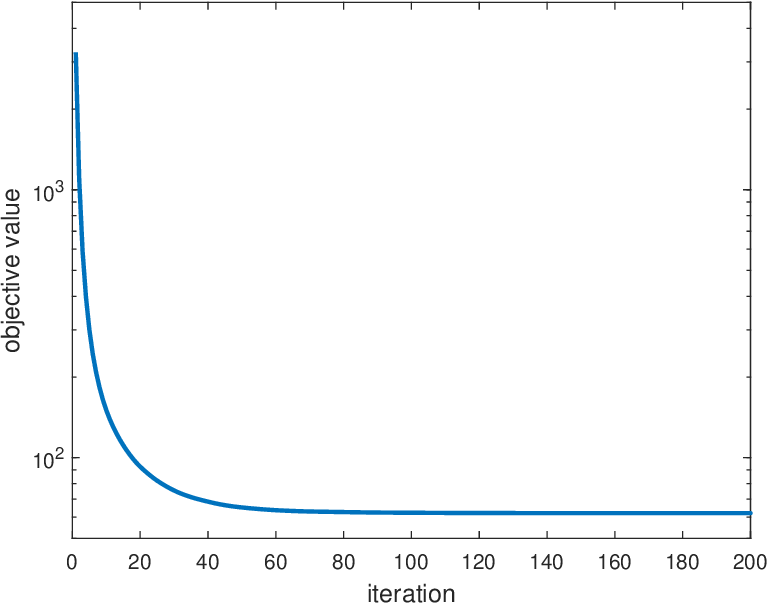}  
			\end{tabular}
		\end{center}
		\caption{The objective function \eqref{eq:general-obj} respect to the iteration counts: $L_1/L_2$ and $L_1/S_K$ for signal recovery (left) and $L_1/L_2$ on the gradient for image recovery (right). The decay in each objective function provides empirical evidence of the convergence of the proposed scheme  \eqref{eq:minimization_flow}.
		}\label{fig:obj}
	\end{figure}

The convergence analysis we conduct in Section \ref{sect: convergence} is based on a modified model \eqref{eq:minimization_flow_new}, as opposed to our numerical scheme \eqref{eq:minimization_flow}. Here we empirically demonstrate the convergence of the latter on the three  quotient models: $L_1/L_2$, $L_1/S_K$, and $L_1/L_2$ on the gradient. The first two models are related to signal recovery, and we choose $K= 100$ for the $L_1/S_K$ model in this experiment, while the last one is stemmed from the image processing literature. The objective function for all these models is expressed in \eqref{eq:general-obj}. We plot the objective value $R(u^k)+\frac{\lambda}{2}\|Au^k-f\|_2^2$ with respect to $k$,
in which $u^k$ is defined by \eqref{eq:minimization_flow}.
As illustrated in Figure \ref{fig:obj}, all the objective curves decrease rapidly, which provides strong evidence that the proposed scheme \eqref{eq:minimization_flow} is convergent. The theoretical analysis of \eqref{eq:minimization_flow} is left for future work.

\begin{figure}
		\begin{center}
			\begin{tabular}{cc}
			   Case 1 &  Case 2 \\
				\includegraphics[width=0.42\textwidth]{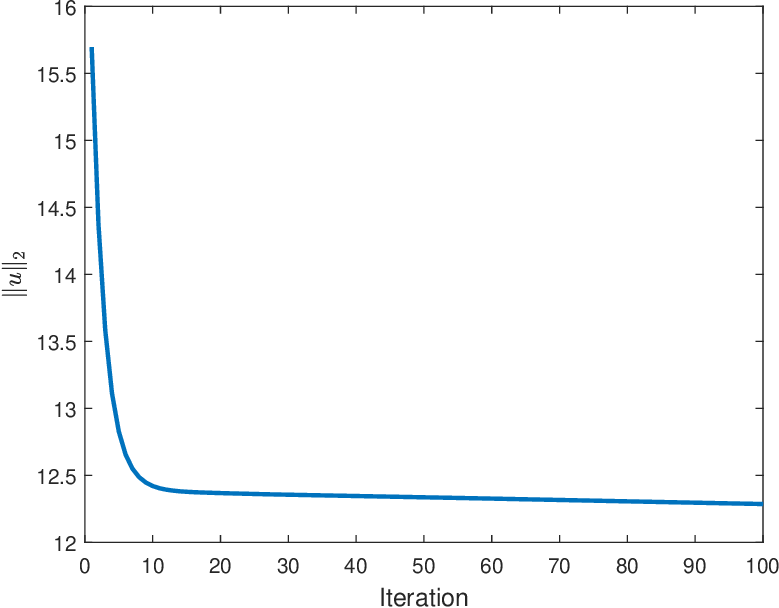} &
			    \includegraphics[width=0.42\textwidth]{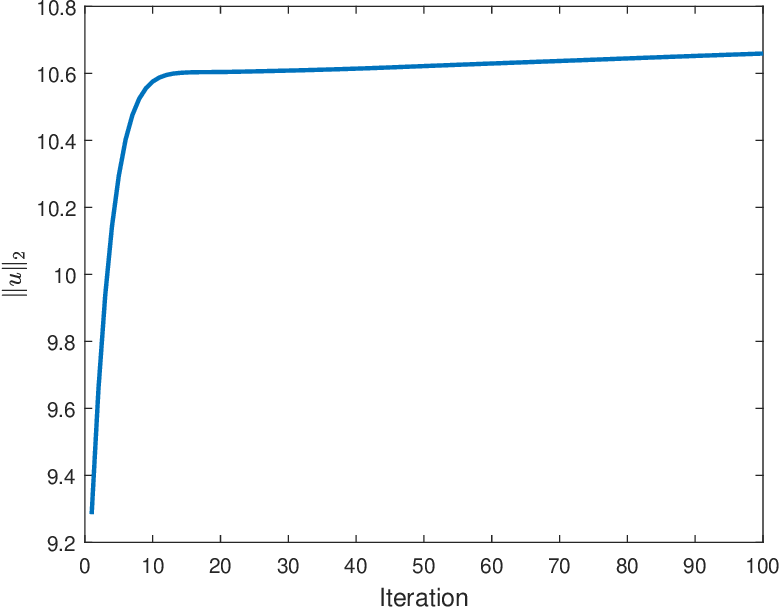}  \\
       \includegraphics[width=0.42\textwidth]{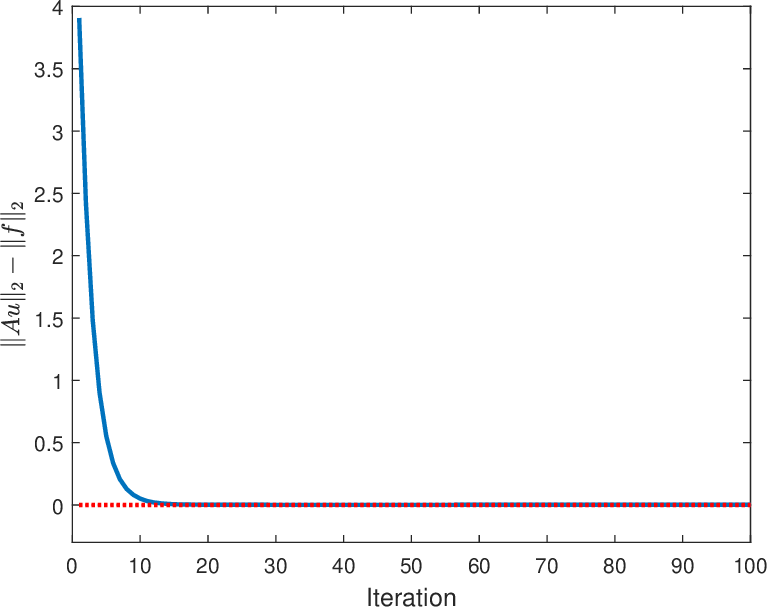} &
			    \includegraphics[width=0.42\textwidth]{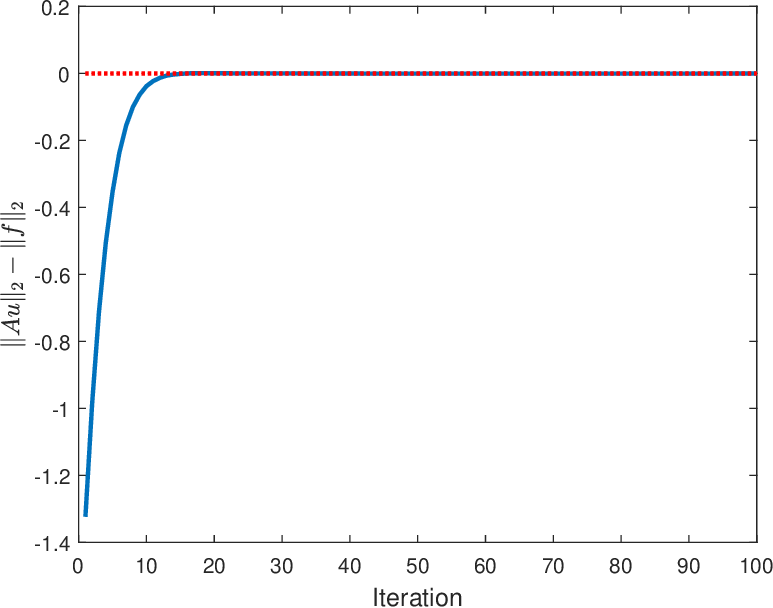}  \\
			\end{tabular}
		\end{center}
		\caption{Numerical verification for Theorem \ref{thm:flow29} based on the $L_1/L_2$ model: if $\|Au\|\geq \|f\|$, then $\|u\|_2$ decreases with respect to the iteration (left); otherwise, $\|u\|_2$ increases (right). We plot $\|u\|_2$ on the top row, while $\|Au\|_2-\|f\|_2$ with a baseline of 0 (red dash line) on the bottom row. 
		}\label{fig:unorm}
	\end{figure}

Next, we numerically verify Theorem \ref{thm:flow29} based on the $L_1/L_2$ model. Specifically, we choose an initial guess of $u^0$ such that $\|Au^0\|_2-\|f\|_2$ is strictly larger than 0 as Case 1, and $\|Au^0\|_2-\|f\|_2<0$ as Case 2. We plot $\|u^k\|_2$ and $\|Au^k\|_2-\|f\|_2$ with respect to $k$ in Figure~\ref{fig:unorm}, which validates the decrease in $\|u\|_2$ is attributed to $\|Au^k\|_2\geq \|f\|_2$. 

Lastly, we investigate the impact of the parameter $K$ for the $L_1/S_K$ model.  We consider $m=250$ to $360$ with an increment of 10. For each $m$, we generate a random matrix $A$, a ground-truth sparse vector $u$ of $s=130$ nonzero elements, and a noise term $\nu$ to obtain the measurement vector $f$. We conduct  100 random realizations and record   in  Table \ref{Tab:model_k} the average value of MSEs between the the ground-truth $u$ and reconstructed solutions by $L_1/S_K$ with $K=10, 100, 150, n(=512)$. We use the $L_1$ solution as the initial condition for $L_1/S_K$, which is referred to as the baseline model in Table \ref{Tab:model_k}. Notice that for $K= n$, the $L_1/S_K$ model becomes  $L_1/L_2$. 
Table~\ref{Tab:model_k} shows that the $L_1/S_K$ model exhibits a close approximation to the oracle performance when $K = 100$ or 150, as the ground-truth sparsity is 130. 
When the parameter $K$ is close to the ground-truth level,  $L_1/S_K$ achieves top-notch performance at any $m$. For a smaller value of $m$, the problem becomes more ill-posed, and  hence all models lead to similar performance. If we choose $K = 10$ (far away from the true sparsity), the performance of $L_1/S_K$ is worse than the $L_1/L_2$ model, which implies that $K$ plays an important role in the success of the $L_1/S_K$ model for sparse recovery.

\begin{table}
		\begin{center}
			\caption{Impact of the parameter $K$ on the sparse recovery via  the $L_1/S_K$ model. The sensing matrix $A$ is of size $m \times n$, where $m$ ranges from 250 to 360 and $n = 512$. The ground-truth sparse vector contains $s=130$ nonzero elements.  Each recorded value is averaged over 100 random realizations. The baseline model refers to the $L_1$ minimization, whose solution serves as the initial condition for $L_1/S_K$. When $K$ is chosen to be close to the true sparsity level (e.g., $K=100, 150$ versus $s=130$), $L_1/S_K$ yields top-notch performance; otherwise (e.g., $K=10$), $L_1/L_2 (K=n)$ is the best.} 
		\begin{tabular}{|l|cccccc|} 
				\hline 
			 \diagbox[width=\dimexpr 0.5\textwidth/8+4\tabcolsep\relax, height=0.8cm]{$K$}{$m$}  & 250 & 260 & 270 &280 &290&  300 \\ \hline
		    baseline &  5.27   & 4.97   & 4.59   & 4.44   & 4.20   & 3.91  \\
               10 & 5.20 & 4.87 & 4.44 & 4.19 & 3.93 & 3.69    \\
               100 & 4.95 & 4.57 & 4.12 & \bf 3.80 & 3.56 & \bf 3.29 \\
		      150  & \bf  4.92 & \bf  4.55 & \bf  4.10 &\bf  3.80 & \bf 3.53 & \bf 3.29  \\
        $n$ & 5.01 & 4.65 & 4.19 & 3.90 & 3.65 & 3.43 \\ \hline
	       	 	\hline
                 \diagbox[width=\dimexpr 0.5\textwidth/8+4\tabcolsep\relax, height=0.8cm]{$K$}{$m$}  &310 &320 & 330 &340 &350&   360 \\ \hline
		    baseline  & 3.73   & 3.55   & 3.49   & 3.26   & 3.13   & 3.02 \\ 
                10 & 3.42 & 3.25 & 3.11 & 2.97 & 2.86 & 2.75    \\
                100  & \bf 3.01 & \bf 2.86 & \bf 2.70 & \bf 2.57 & \bf 2.46 & \bf 2.34 \\
		      150  & 3.02 & \bf 2.86 & 2.71 & 2.58 & 2.49 & 2.39    \\ 
        $n$ & 3.16 & 3.04 & 2.91 & 2.81 & 2.73 & 2.65
        \\ \hline
			\end{tabular}\label{Tab:model_k}
			\medskip
		\end{center}
	\end{table}
 
\subsection{Signal recovery}\label{sect:exp-signal}

This section investigates the signal recovery problem, in which we compare various algorithms for the QRM model together with 
fractional programming (FP). Specifically for QRM, we compare
the proposed Algorithm \ref{alg:unified_signal}
on both $L_1/L_2$  and $L_1/S_K$ (choosing $K=100$) regularizations with a difference of convex algorithm (DCA) scheme \cite{phamLe2005dc} implemented by ourselves. 
Here DCA aims to minimize $D_1(u)-D_2(u)$ with convex functionals $D_1, D_2$  by iteratively constructing two sequences $\left\{u^{k}\right\}$ and $\left\{v^{k}\right\}$ in the following way,
\begin{equation}\label{eq:dca}
    \left\{\begin{array}{l}
v^{k} \in \partial D_2\left(u^{k}\right) \\
u^{k+1}=\arg \min _{u} D_1(u)-\left\langle u, v^{k}\right\rangle.
\end{array}\right.
\end{equation}
 We consider splitting the objective function \eqref{eq:general-obj} into : 
\begin{equation}
    \begin{split}
        D_1(u) &= \mu \|u\|_1 +  \frac{\lambda}{2} \|Au-  f\|_2,\\
        D_2(u) & = \mu \|u\|_1-R(u).
    \end{split}
\end{equation}
The $u$-subproblem in DCA \eqref{eq:dca} amounts to an $L_1$ regularized problem, which can be solved by ADMM. 

The FP formulation \eqref{eq:l1/Sk-frac} is defined for $L_1/S_K,$ which becomes $L_1/L_2$ for $K=n.$
We compare to a proximal-gradient-subgradient algorithm with backtracked extrapolation (PGSA\_BE) \cite{li2022proximal} for solving \eqref{eq:l1/Sk-frac}.
In addition, we  implement the ADMM  algorithm for the $L_1/L_2$ model under either FP or QRM setting.

We randomly generate the matrix $A$ of size $m\times 512$ for $m$ varying from 240 to 360 with an increment of 20. 
Since the quotient models are non-convex, the choice of initial guess $u^0$ significantly impacts the performance. We adopt the restored solution via the $L_1$ minimization as the initial guess and terminate the iterations when the relative error $\| u^{k+1}-u^{k}\|_2/\|u^{k+1}\|_2$ is less than $10^{-8}$. This 
 stop criterion is used for all the algorithms.
 Table \ref{Tab:algo} reports the averaged MSE values over 100 random realizations.
We observe that the QRM framework always performs better
than FP for the same regularization.
The $L_1/S_K$ model solved by our algorithm performs the best in all the cases when $K=100$ is chosen near the true sparsity level (130), and $L_1/L_2$ without knowing the sparsity ranks the second best.
In short, the proposed algorithms for solving two QRM models with $L_1/L_2$ and $L_1/S_K$ outperform the other relevant approaches.

\begin{table}[t]
		\begin{center}
			\caption{MSEs of recovering a sparse vector of length $n=512$ with $s=130$ nonzero elements from $m$ noisy measurements ($m=240:20:360$ following the MatLab's notation). We compare $L_1/L_2$ and $L_1/S_K$ for $K=100$ under the settings of FP \eqref{eq:l1/Sk-frac}  and QRM \eqref{eq:l1/Sk}. We observe QRM is a better framework than FB for sparse recovery. The best results are consistently given by the proposed algorithm for solving the $L_1/S_K$ model when the value of $K=100$ is close to the true sparsity level ($130$). The $L_1/L_2$ (when $K=n$) model achieves the second best in performance.
   } 
		\begin{tabular}{l|l|ccccccccccc} 
				\hline 
			& model-algorithm & 240 & 260  &280 &  300  &320 &340 &360\\ \hline
\multirow{3}{*}{FP} & $L_1/L_2$-ADMM    & 5.51 & 4.76 & 4.00 & 3.48 & 3.15 & 2.86 & 2.67   \\
      & $L_1/L_2$-PGSA\_BE  & 11.12 & 8.60 & 6.28 & 4.40 & 3.37 & 2.83 & 2.52\\
      & $L_1/S_K$-PGSA\_BE & 5.82 & 4.92 & 4.05 & 3.44 & 3.02 & 2.77 & 2.60\\ \hline
     \multirow{5}{*}{QRM}          
     & $L_1/L_2$-DCA & 5.56 & 4.87 & 4.14 & 3.61 & 3.27 & 2.95 & 2.69\\ 
              & $L_1/L_2$-ADMM  & 5.53 & 4.75 & 3.96 & 3.45 & 3.12 & 2.86 & 2.68
              \\
              & $L_1/L_2$-proposed & 5.50 & 4.70 & 3.92 & 3.40 & 3.07 & 2.81 & 2.64 \\
          & $L_1/S_K$-DCA & 5.52 & 4.77 & 4.01 & 3.48 & 3.15 & 2.86 & 2.67     \\
		     & $L_1/S_K$-proposed   & {\bf 5.44}  & {\bf 4.65} & {\bf 3.83}  & {\bf 3.26} & {\bf 2.91} & {\bf 2.57} & {\bf 2.33}    \\ \hline

			\end{tabular}\label{Tab:algo}
			\medskip
		\end{center}
	\end{table}

		


\subsection{Image recovery}\label{sect:exp-image}

We consider an MRI reconstruction as a proof-of-concept example in image processing. 
 The MRI measurements are acquired through multiple radical lines in the frequency domain, achieved by performing the Fourier transform.
 In addition, we add the Gaussian noise, with a mean of zero and standard deviation $\sigma$ on the MRI measurements. 
 Intuitively, fewer radial lines and a larger $\sigma$ value bring more ill-posedness and difficulty to the problem. 
 Here we consider two standard phantoms, namely Shepp–Logan (SL) phantom  generated using MATLAB's built-in command {\tt phantom} and the FORBILD (FB) phantom \cite{FB_ph}. 
 We evaluate the performance in terms of the relative error (RE) and the peak signal-to-noise
ratio (PSNR), defined by
\begin{equation*}
		\text{RE}( u^\ast,\tilde{ u}) := \frac{\|u^\ast-\tilde{ u}\|_2}{\|\tilde{ u}\|_2} \quad \text{and} \quad 	    {\rm PSNR}( u^\ast, \tilde{ u}) := 10 \log_{10} \frac{N P^2}{\| u^\ast-\tilde{ u}\|_2^2},
	\end{equation*} 
	where $ u^\ast$ is the restored image, $\tilde{ u}$ is the ground truth, and $P$ is the maximum peak value of $\tilde{u}. $

Similar to the signal-recovering experiments, we regard the performance of the $L_1$ on the gradient, i.e., the total variation (TV), as the baseline. 
For $L_1/L_2$ on the gradient, we compare the proposed  algorithm to a previous method based on ADMM  \cite{rahimi2019scale}. For three sampling schemes (7, 10, and 13  lines) and two noise levels ($\sigma = 0.01$ and 0.05), we record RE and PSNR values of three methods in Table \ref{Tab:MRI}, demonstrating significant improvements in the accuracy of the proposed approach over the previous works. 

Figures \ref{fig:phantomSL} and \ref{fig:phantomFB} present visual reconstruction results of the SL phantom and the FB phantom, respectively, both under high additive Gaussian noise ($\sigma=0.05$). In particular, Figure \ref{fig:phantomSL} is to recover the SL phantom using 7 radial lines. The $L_1$ model has severe streaking artifacts due to this extremely small number of data obtained on the radial lines. The $L_1/L_2$ minimization on the gradient yields significant improvements over the baseline model (TV). The proposed algorithm outperforms the previous ADMM approach at the outer ring and boundaries of the three middle oval shapes, which are more obvious in the difference map to the ground truth.
On the other hand, the FB phantom has finer structures and lower image contrast compared to the SL phantom. As a result, it requires 13 radial lines for a reasonable reconstruction. As we observe in Figure \ref{fig:phantomFB}, the overall geometric shapes are preserved. At the same time, many speckle artifacts appear in the reconstructed images by $L_1/L_2$ no matter which algorithm is used.

\begin{table}[h]
		\begin{center}
			\scriptsize
			\caption{MRI reconstruction from   different numbers of radial lines and different noise levels. }
		\begin{tabular}{|c|c|c|cc|cc|cc|} 
				\hline 
				\multirow{2}{*}{Image} & \multirow{2}{*}{$\sigma$} & \multirow{2}{*}{Line} & \multicolumn{2}{c|}{$L_1$}& \multicolumn{2}{c|}{$L_1/L_2$-ADMM } & \multicolumn{2}{c|}{$L_1/L_2$-proposed }  \\ \cline{4-9} 
				& & &   RE & PSNR    & RE & PSNR    & RE & PSNR  \\ \hline
			\multirow{6}{*}{SL} & \multirow{3}{*}{0.01}& 7 &  46.06\% & 19.50   & 25.36\% & 24.09  &  {\bf 3.74\%}  &  {\bf 40.72}  \\ \cline{3-9} 
				& &  10  & 16.29\% & 28.66  & 3.41\% & 41.53 
    &  {\bf 2.91\% } &  {\bf 42.90}    \\ \cline{3-9}   
				& &  13 &   6.85\% & 36.52  & 1.91\% & 46.55  &  {\bf 1.71}\% &  {\bf 47.49}  
				 \\ \cline{2-9} 
					& \multirow{3}{*}{0.05}
     & 7 &  52.31\% & 18.33   & 43.63\% & 19.38 &  {\bf 31.90\%}  &  {\bf 22.10}   \\ \cline{3-9} 
     & & 10 &  33.09\%   & 22.42 & 14.34\% & 29.04 &   {\bf 14.08\%} &  {\bf 29.24}   \\ \cline{3-9} 
				& &  13  & 22.67\% & 26.10 
    &   10.50\%   &   31.75 &   {\bf 10.41\%} & {\bf 31.82}  \\ \hline
    		\multirow{6}{*}{FB} & \multirow{3}{*}{0.01}& 7 &  21.63\% & 21.49   & 13.80\% & 24.89   &  {\bf 1.11\%}  &  {\bf 26.94}  \\ \cline{3-9} 
				& & 10 &  18.14\%   & 23.08 & 14.98\% & 24.17 &    {\bf 12.90\%} &  {\bf 25.47}  \\ \cline{3-9}  
				& &  13 &   9.51\% & 28.29 &  1.41\% & 44.71 &   {\bf 1.17}\% &  {\bf 46.31}  
				 \\ \cline{2-9} 
					& \multirow{3}{*}{0.05}
     & 7 &  26.03\% & 19.9   & 22.14\% & 20.78 &   {\bf 16.50\%}  &  {\bf 23.36}  \\ \cline{3-9} 
     & & 10 &  18.14\%   & 23.08 & 14.98\% & 24.17 &    {\bf 12.90\%} &  {\bf 25.47}  \\ \cline{3-9} 
				& &  13  & 14.48\% & 24.79 
    &   12.67\%   &   25.64   & {\bf 12.30\%} & {\bf 25.89}  \\ \hline
			\end{tabular}\label{Tab:MRI}
			\medskip
		\end{center}
	\end{table}

\begin{figure}[h]
		\begin{center}
			\begin{tabular}{ccc}
				$L_1$ & $L_1/L_2$-ADMM & $L_1/L_2$-proposed \\
				
    \includegraphics[width=0.305\textwidth]{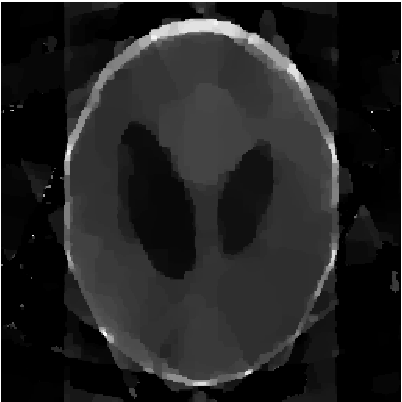} &
				\includegraphics[width=0.305\textwidth]{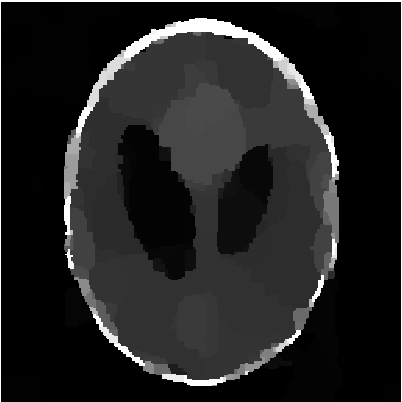} &
				\includegraphics[width=0.305\textwidth]{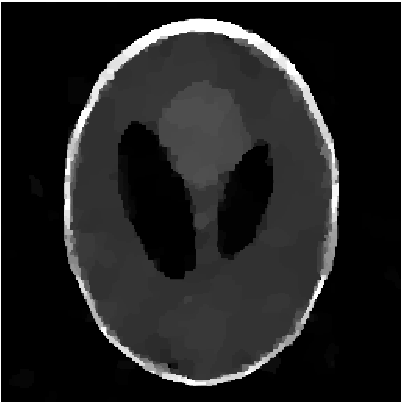} \\
        \includegraphics[width=0.305\textwidth]{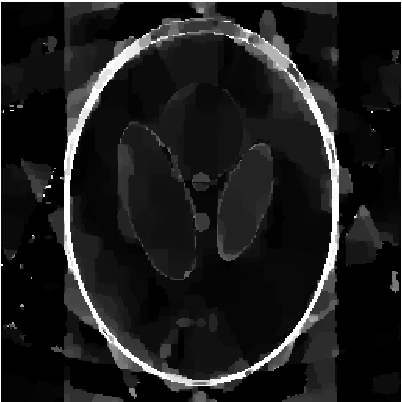} &
				\includegraphics[width=0.305\textwidth]{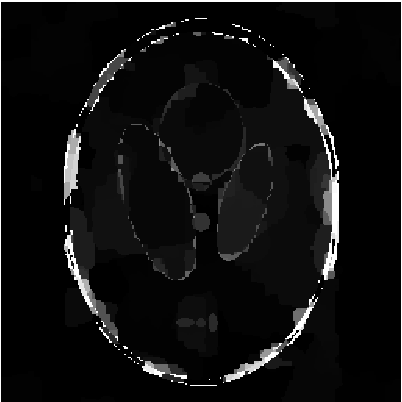} &
				\includegraphics[width=0.305\textwidth]{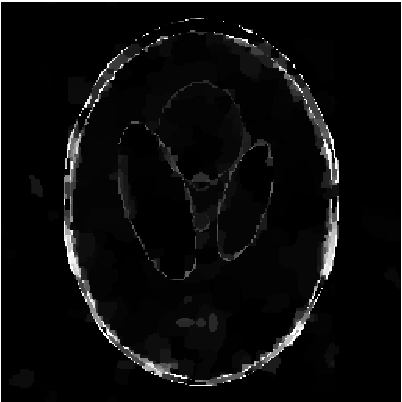} \\
			\end{tabular}
		\end{center}
		\caption{MRI reconstruction on 
  the SL phantom with a noise level of $\sigma=0.05$ with 7 radial lines. 
  Top row -- reconstruction results, bottom row -- difference from ground truth.
  The proposed algorithm outperforms the previous ADMM approach at the outer ring and boundaries of the three middle oval shapes, better seen in the difference map.
  }\label{fig:phantomSL}
	\end{figure}

\begin{figure}[h]
		\begin{center}
			\begin{tabular}{ccc}
				$L_1$ & $L_1/L_2$-ADMM & $L_1/L_2$-proposed \\

    \includegraphics[width=0.305\textwidth]{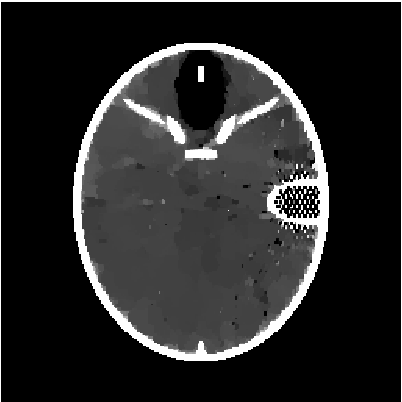} &
				\includegraphics[width=0.305\textwidth]{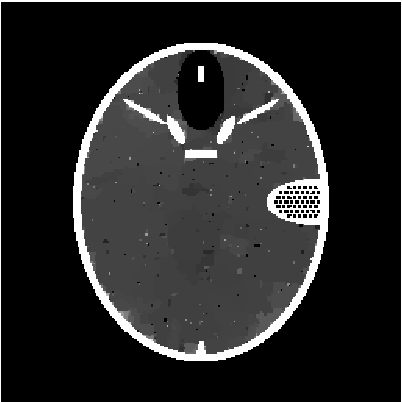} &
				\includegraphics[width=0.305\textwidth]{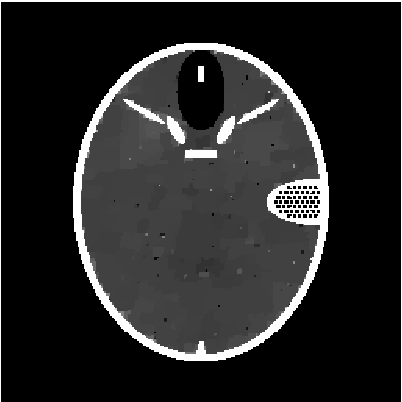} \\
        \includegraphics[width=0.305\textwidth]{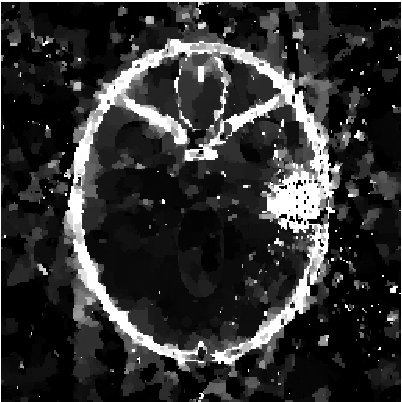} &
				\includegraphics[width=0.305\textwidth]{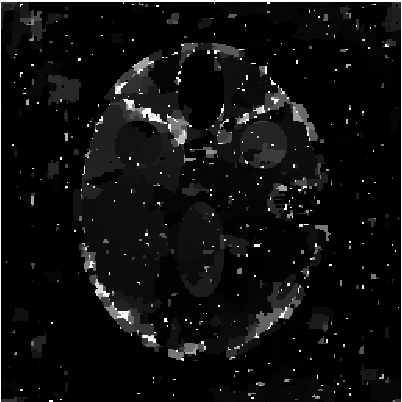} &
				\includegraphics[width=0.305\textwidth]{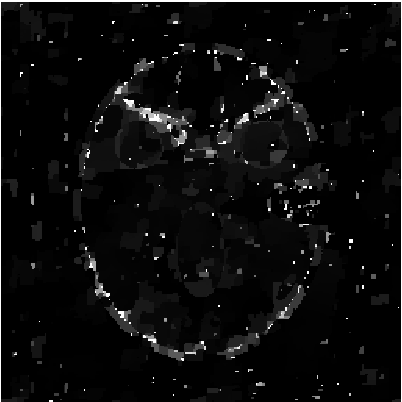} 
			\end{tabular}
		\end{center}
		\caption{MRI reconstruction on the  FB phantom with a noise level of $\sigma=0.05$ with 13 radial lines. Top row -- reconstruction results, bottom row -- difference from ground truth.  The proposed algorithm is able to better preserve the overall geometric shapes, compared to competing methods. 
  }\label{fig:phantomFB}
	\end{figure}

\section{Conclusions}
\label{sect:conclusion}
In this paper, we proposed a gradient descent flow to minimize a quotient regularization model with a quadratic data fidelity term for signal and image processing applications. We assumed the numerator and the denominator in the quotient model are absolutely one homogeneous, which enables us to establish the convergence in a continuous formulation. By taking the implementation details into consideration, we adopted a slightly different discretized scheme to the one we analyze theoretically. The proposed algorithm amounts to solving a convex problem iteratively. Experimentally, we presented the comparison results of three case studies of $L_1/L_2$ and $L_1/S_K$ for signal recovery and $L_1/L_2$ on the gradient for MRI reconstruction. We demonstrated that the proposed algorithm significantly outperforms the previous methods in each case in terms of accuracy. Future work includes the speed-up of the proposed algorithm, e.g., trying to make a single loop rather than the double loop, and the convergence analysis of the actual scheme.

\begin{acknowledgements}
 C.~Wang was partially supported by the Natural Science Foundation
of China (No. 12201286), HKRGC Grant No.CityU11301120, and the Shenzhen Fundamental Research Program JCYJ20220818100602005. Y.~Lou was partially supported by NSF CAREER award 1846690.
J-F.~Aujol and G.~Gilboa acknowledge the support of the European Union’s Horizon 2020 research and innovation program under the Marie Sklodowska-Curie grant agreement No777826. G.~Gilboa acknowledges support by ISF grant 534/19. This work was initiated while J-F.~Aujol and Y.~Lou were visiting the Mathematical Department of UCLA.
\end{acknowledgements}

\section*{Data Availability}
The MATLAB codes and datasets generated and/or analyzed during the current study  will be available after publication.

\section*{Declarations}
The authors have no relevant financial or non-financial interests to disclose.  The authors declare that they have no conflict of interest.

%
%

\bibliographystyle{spmpsci}      
\bibliography{refer_l1dl2}


\end{document}